\documentclass[11pt]{amsart}
\usepackage{microtype,fullpage,amsmath,amsthm}
\usepackage{graphicx,setspace,amscd,float,hyperref,enumitem, bbold}

\newtheorem{thm}{Theorem}[section]
\newtheorem{lem}[thm]{Lemma}

\newtheorem{prop}[thm]{Proposition}
\newtheorem{cor}[thm]{Corollary}
\theoremstyle{definition}
\newtheorem{df}[thm]{Definition}
\newtheorem{rk}[thm]{Remark}

\newtheorem{nt}[thm]{Notation}
\newtheorem{con}[thm]{Convention}

\newtheorem*{mainthmA}{Theorem A}
\newtheorem*{mainthmB}{Theorem B}

\newcommand{\Gam}{\Gamma}

\newcommand{\ol}{\overline}
\newcommand{\mG}{\mathcal{G}}
\newcommand{\mI}{\mathcal{I}}
\newcommand{\mA}{\mathcal{A}}
\newcommand{\mE}{\mathcal{E}}

\newcommand{\mD}{\mathcal{D}}

\newcommand{\mP}{\mathcal{P}}

\newcommand{\mW}{\mathcal{W}}
\newcommand{\EG}{\mathcal{E}(\Gamma)}
\newcommand{\iwp}{\mathcal{IW}(\vphi)}
\newcommand{\swg}{\mathcal{SW}(g)}

\newcommand{\cwg}{\mathcal{CW}(g)}
\newcommand{\id}{\mathcal{ID}}
\newcommand{\iw}{\mathcal{IW}}
\newcommand{\idg}{\mathcal{ID}(\mathcal{G})}
\newcommand{\ar}{\mathcal{A}_r}
\newcommand{\phar}{\vphi \in \mathcal{A}_r}

\newcommand{\vphi}{\varphi}

\newcommand{\outr}{Out(F_r)}
\newcommand{\autr}{Aut(F_r)}

\newcommand{\Z}{\mathbb Z}

\begin{document}

\title{Ideal Whitehead Graphs in $\outr$ IV:\\ Building ideal Whitehead graphs in higher ranks and ideal Whitehead graphs with cut vertices}
\author{Catherine Pfaff}

\address{\tt Department of Mathematics, University of California at Santa Barbara 
  \newline \indent  {\url{http://math.ucsb.edu/~cpfaff/}}, } \email{\tt cpfaff@math.ucsb.edu}

\date{}
\maketitle

\begin{abstract}
We provide an example in each rank of an ageometric fully irreducible outer automorphism whose ideal Whitehead graph has a cut vertex. Consequently, we show that there exist examples in each rank of axis bundles (in the sense of \cite{hm11}) that are not just a single axis, as well as of ``nongeneric'' behavior in the sense of \cite{kp15}. The tools developed here also allow one to construct a whole new array of ideal Whitehead graphs achieved by ageometric fully irreducible outer automorphisms in all ranks.
\end{abstract}

\section{Introduction}

Let \emph{$F_r$} denote the rank-$r$ free group and $Out(F_r)$ its outer automorphism group. The series of papers \cite{IWGI}, \cite{IWGII}, \cite{IWGIII} provide realization results for an invariant dependent only on the conjugacy class (within $Out(F_r)$) of the outer automorphism, namely the ``ideal Whitehead graph.'' In this paper we develop tools for using the previous realization results to give new, higher-rank, results. In particular, the tools will allow us to find a fully irreducible in each rank whose ideal Whitehead graph has a cut vertex. By \cite{mp13}, this gives an example in each rank of a fully irreducibles whose axis bundle is more than just a single axis. It further gives examples of ``nongeneric'' behavior in the sense of \cite{kp15}.

A ``fully irreducible'' (iwip) outer automorphism is the most commonly used analogue to a pseudo-Anosov mapping class and is generic. An element $\vphi \in Out(F_r)$ is \emph{fully irreducible} if no positive power $\vphi^k$ fixes the conjugacy class of a proper free factor of $F_r$. 

We give in Subsection \ref{ss:iwg} the exact $Out(F_r)$ definition of an ideal Whitehead graph. For now, we just remark that it is a finer invariant than the $Out(F_r)$ analogue of the index list for a pseudo-Anosov surface homeomorphism. In particular, it provides further detailed information than the index list does about the expanding lamination for a fully irreducible. Perhaps most relevant to our aims is that the ideal Whitehead graph gives information about the geometry of the axis bundle $\mA_{\vphi}$ of a fully irreducible $\vphi \in Out(F_r)$

As in the previous papers of this series, we focus on the situation of graphs with $2r-1$ vertices because this restricts our attention to fully irreducibles as close to those coming from surface homeomorphisms (\emph{geometrics}) as possible without actually coming from surface homeomorphisms (or being geometric-like \emph{parageometrics}).

Mosher and Handel defined in \cite{hm11} the axis bundle and used ideal Whitehead graphs to prove that the axis bundle for each nongeometric fully irreducible outer automorphism is proper homotopy equivalent to a line. The axis bundle is an analogue of the axis for a hyperbolic isometry of hyperbolic space or the Teichm\"{u}ller axis for a pseudo Anosov. We give a formal definition in Subsection \ref{ss:AxisBundle}. One important fact about the axis bundle is that, if $\vphi$ and $\psi$ are two fully irreducible outer automorphisms, then $\mathcal{A}_{\vphi}$ and $\mathcal{A}_{\psi}$ differ by the action of an $Out(F_r)$ on its Culler-Vogtmann outer space if and only if  $\vphi$ and $\psi$ have powers conjugate in $Out(F_r)$. More recently Mosher and Pfaff \cite{mp13} proved a necessary and sufficient ideal Whitehead graph condition for an axis bundle to be a unique, single axis. \cite{IWGII} gives examples in each rank where this occurs. \cite{kp15} then shows that this behavior is generic along a particular ``train track directed'' random walk. We use the tools developed in this paper to prove that, even along the ``train track directed'' random walk, one still obtains fully irreducibles with a cut vertex in their ideal Whitehead graph and hence more than one axis in their axis bundle.

The following is our main theorem:

\begin{mainthmA}\label{main1}
For each $r \geq 3$, there exists an ageometric fully irreducible $\vphi \in Out(F_r)$ such that the ideal Whitehead graph $\mathcal{IW}(\vphi)$ is a connected $(2r-1)$-vertex graph with at least one cut vertex. In particular, there exists an ageometric fully irreducible $\vphi \in Out(F_r)$ with the single-entry index list $\{\frac{3}{2}-r\}$ and such that the axis bundle of $\vphi$ has more than one axis.
\end{mainthmA}

\subsection*{Acknowledgements}

\indent The author expresses gratitude to Ilya Kapovich and Lee Mosher for their continued support and interest in her work.

\section{Preliminary definitions and notation}{\label{Ch:PrelimDfns}}

We will provide methods for constructing, for some particular graphs $\mathcal{G}$, a train track representative (in the sense of \cite{bh92}) of a fully irreducible $\vphi \in Out(F_r)$ with ideal Whitehead graph $\mathcal{IW}(\vphi) \cong \mathcal{G}$. Subsection \ref{ss:tt} is thus devoted to establishing the definitions and notation we use regarding train track maps. The existence of ``ideally decomposed'' representatives was established in \cite{IWGII}, see Subsection \ref{ss:id}, allowing us to restrict our attention to maps on roses. In \cite{IWGII}, we introduced ltt structures (described in Subsection \ref{ss:ltt}), graphs extending a potential ideal Whitehead graph to include the edges of the rose, as well as an edge including the nonfixed direction of the train track representative. In a sense they are formed by ``blowing up'' the vertex and inserting an extension of the ideal Whitehead graph. An ``admissible fold'' (in practice an inverse of a fold) yields a move between ltt structures, either an ``extension'' or ``switch,'' as described in Subsection \ref{ss:se}. For a potential ideal Whitehead graph $\mathcal{G}$, an ``$\mathcal{ID}$ diagram'' is defined in \cite{IWGI} whose vertices are ltt structures, whose directed edges are admissible moves, and who contain loops for the ideally decomposed train track representatives of fully irreducible $\vphi \in Out(F_r)$ with $\mathcal{IW}(\vphi) \cong \mathcal{G}$. We describe $\id$ diagrams in Subsection \ref{ss:idd}. We explain in Subsection \ref{ss:iNPprevention} a tool (developed in \cite{IWGII}) for ensuring our maps do not have periodic Nielsen paths (see Subsection \ref{ss:pnp}). In Subsection \ref{ss:AxisBundle} we give an introduction to the Handel-Mosher axis bundle, which our main theorem indicates is not always a single axis, even with quite strong invariant value restrictions.

\begin{con}
Given a rank $r \geq 2$, we let $\mathcal{FI}_r$ denote the set of all fully irreducible $\vphi \in Out(F_r)$. We also fix once and for all a free basis $X=\{X_1, \dots, X_r\}$ for $F_r$. Since every map we deal with in this paper is on the rose (wedge of $r$ circles) $R_r$, we assume $\Gamma$ is the rose, with a single vertex $v$, and give all definitions only in this restrictive setting. A marking for us will be a homotopy equivalence $R_r \to \Gamma$.
\end{con}

\subsection{Graph maps and train track maps}{\label{ss:tt}}
%We recall from \cite{bh92} the definition of a ``train track representative'' for a $\vphi \in Out(F_r)$. Let $R_r$ denote the $r$-petaled rose (graph with one vertex and $r$ edges) together with an identification $\pi(R_r) \cong F_r$. A connected $1$-dimensional CW-complex $\Gamma$ such that each vertex has valence $\geq 2$, together with a homotopy equivalence (\emph{marking}) $R_r \to \Gamma$, is called a \emph{marked graph}. A \emph{train track (tt) representative} of $\vphi$ is a homotopy equivalence $g \colon \Gamma \to \Gamma$ of a marked graph $\Gamma$, where $\vphi=g_*$, where $g$ sends vertices to vertices, and where $g^k$ is locally injective on edge interiors for each $k>0$.

%Every map we deal with in this paper is on the rose. Thus, we assume $\Gamma$ is a rose (wedge of circles), with vertex $v$, and give all definitions only in this restrictive setting. 

\begin{df}[Graph map and train track map]
We call a continuous map $g \colon \Gamma \to \Gamma$ a \emph{graph map} if it takes $v$ to $v$ and is locally injective on edge interiors. Here we assume a graph map is also a homotopy equivalence. A graph map $g$ is a \emph{train track map} if additionally each $g^k$, with $k > 1$, is locally injective on edge interiors (and generally also that $\Gamma$ comes with a marking). 
\end{df}

In general we use definitions of \cite{bh92} and \cite{bfh00} for discussing train track maps. We remind the reader here of some additional definitions and notation given in \cite{IWGI}.

We denote by $\mE^+(\Gamma):= \{E_1, \dots, E_{n}\}$ the edge-set of $\Gamma$ with a prescribed orientation and set
$\mE(\Gamma):=\{E_1, \overline{E_1}, \dots, E_n, \ol{E_n} \}$, where $\ol{E_i}$ is $E_i$ oppositely oriented. If an indexing of $\mE^+(\Gamma)$ (thus of $\EG$) is prescribed, we call $\Gamma$ \emph{edge-indexed}. 

Depending on the context, an \emph{edge-path} $e_1 \cdots e_k$ of \emph{length} $n$ will either mean a continuous map $[0,n] \to \Gam$ that, for each $1 \leq i \leq k$, maps $(i-1,i)$ homeomorphically to $int(e_i)$, or the sequence of edges $e_1, \dots, e_k \in \mE(\Gamma)$.

\begin{df}[Irreducible and expanding train track maps]
The train track map $g \colon \Gamma \to \Gamma$ is \emph{irreducible} if $g$ leaves invariant no proper subgraph with a noncontractible component. We call it \emph{strictly irreducible} if, for each $E_i, E_j \in \mE^+(\Gamma)$, we have that $g(E_j)$ contains either $E_i$ or $\ol{E_i}$. (Notice that strict irreducibility implies irreducibility and that each irreducible map has a strictly irreducible power). $g$ is \emph{expanding} if for each $E \in \mE^+(\Gamma)$, $length(g^n(E)) \to \infty$ as $n \to \infty$.
\end{df}

\begin{df}[Directions and turns]
$\mD(\Gam)$ will denote the set of \emph{directions} at $v$, i.e. germs of edge segments emanating from $v$. For each $e \in \mathcal{E}(\Gamma)$, we also let $e$ denote the initial direction of $e$. For each nontrivial edge-path $\gamma=e_1 \cdots e_k$ in $\Gamma$, we define $D_0 \gamma := e_1$. $Dg$ will denote the direction map induced by $g$, i.e. $Dg(e)=e_1$ where $g(e)=e_1 \dots e_k$ for some $e_1, \dots, e_k \in \EG$. We call $d \in \mD(\Gamma)$ \emph{periodic} if $Dg^k(d)=d$ for some $k>0$ and \emph{fixed} if $k=1$.

By a \emph{turn} at $v$, we will mean an unordered pair $\{d_i, d_j\}$ with $d_i, d_j \in \mD(\Gam)$. We let $\mathcal{T}(\Gam)$ denote the set of turns at $v$ and $D^tg$ the induced turn map (where $D^tg(\{e_i, e_j\}) = \{Dg(e_i), Dg(e_j)\}$). For an edge-path $\gamma=e_1 \cdots e_k$ in $\Gamma$, we say $\gamma$ \emph{contains} or \emph{takes} $\{\overline{e_i}, e_{i+1} \}$ for each $1 \leq i < k$. A turn $\{d_i, d_j\}$ is \emph{illegal} for $g$ if $Dg^p(d_i)=Dg^p(d_j)$ for some integer $p > 1$ and is considered \emph{legal} otherwise.
\end{df}

\begin{rk}
Let $\vphi \in Aut(F_r)$ be defined for each $X_i \in X$ by $\vphi(X_i)=x_{i,1} \cdots x_{i,n_i}$ for some $x_{i,k} \in X^{\pm 1}$. We define a graph map $g_{\vphi}$ corresponding to $\vphi$. Let $\Gamma$ be the rose $R_r$ with the identity marking. Then, for each $E_i \in \mE^+(\Gam)$, we have $g_{\vphi}(E_i)=e_{i,1} \cdots e_{i,n_i}$ with the $e_{i,k} \in \mE(\Gam)$ such that, under the identification of $F(X_1, \dots, X_r)$ with $\pi_1(R_r, v)$, we have that each $E_i$ and $e_{i,k}$ corresponds to $X_i$ and $x_{i,k}$. By the correspondence of each $\vphi \in Aut(F_r)$ with a graph map $g_{\vphi}$, we can translate all of the language of directions, turns, and so on, to free group automorphisms.
\end{rk}

\subsection{Periodic Nielsen paths}{\label{ss:pnp}}

\begin{df}[Nielsen paths] Let $g \colon \Gamma \to \Gamma$ be an expanding irreducible train track map. Bestvina and Handel \cite{bh92} define a nontrivial tight path $\rho$ in $\Gamma$ to be a \emph{periodic Nielsen path (PNP)} if, for some power $R \geq 1$, we have $g^R(\rho) \cong \rho$ rel endpoints (and just a \emph{Nielsen path (NP)} if $R=1$). A NP $\rho$ is called \emph{indivisible} (hence is an ``iNP'') if it cannot be written as $\rho = \gamma_1\gamma_2$, where $\gamma_1$ and $\gamma_2$ are themselves NPs.
\end{df}

Bestvina and Handel describe in \cite[Lemma 3.4]{bh92} the structure of iNP's:

\begin{lem}[\cite{bh92}]{\label{l:iNP}}
Let $g \colon \Gamma\to\Gamma$ be an expanding irreducible train track map and $\rho$ in $\Gamma$ an iNP for $g$. Then $\rho=\bar \rho_1\rho_2$, where $\rho_1$ and $\rho_2$ are nontrivial legal paths  originating at a common vertex $v$ and such that the turn at $v$ between $\rho_1$ and $\rho_2$ is an nondegenerate illegal turn for $g$. 
\end{lem}

For a fully irreducible $\vphi \in \mathcal{FI}_r$, \cite{bh92} gives an algorithm for finding an expanding irreducible train track representative (\emph{stable} representative) with the minimal number of iNPs.

\begin{df}[Rotationless]  An expanding irreducible train track map is called \emph{rotationless} if each periodic direction is fixed and each PNP is of period one. By \cite[Proposition 3.24]{fh11}, one can define a $\vphi \in \mathcal{FI}_r$ to be \emph{rotationless} if one (hence all) of its train track representatives is rotationless.
\end{df}

Gaboriau, J\"ager, Levitt, and Lustig introduce in \cite{gjll} the ``ageometric'' subclass of nongeometric fully irreducible outer automorphisms, defined there in terms of the index sum. We give here an equivalent definition in terms of periodic Nielsen paths.

\begin{df}[Ageometrics] An outer automorphism $\vphi \in \mathcal{FI}_r$ is called \emph{ageometric} if it has a train track representative with no PNPs (equivalently, each stable representative of each rotationless power has no PNPs).
\end{df}

\begin{con}[$\mA_r$]
Since ageometrics are our main focus, we set notation and let $\ar$ denote the subset of $\mathcal{FI}_r$ consisting of the ageometric elements of $\mathcal{FI}_r$. In other words, an element of $\mA_r$ is an ageometric fully irreducible in $Out(F_r)$.
\end{con}

\subsection{Whitehead graphs}{\label{ss:iwg}}

Ideal Whitehead graphs were introduced by Handel and Mosher in \cite{hm11}. They are proved an outer automorphism invariant in \cite{p12a}, where the equivalence of several ideal Whitehead graph definitions is also established. We give here an explanation in terms of a PNP-free train track representative $g \colon \Gamma \to \Gamma$ of a $\vphi \in \mathcal{A}_r$, where $\Gamma$ is a rose with vertex $v$.

\begin{df}[Local Whitehead graphs $\mathcal{LW}(g)$]
The \emph{local Whitehead graph} $\mathcal{LW}(g)$ for $g$ at $v$ has:
\begin{enumerate}
\item a vertex for each direction $d \in \mathcal{D}(\Gam)$ and
\item edges connecting vertices for $d_1, d_2 \in \mathcal{D}(\Gam)$ when $\{d_1, d_2\}$ is taken by $g^k(e)$ for some $e \in \mathcal{E}(\Gamma)$ and $k > 1$.
\end{enumerate}
\end{df}

\begin{df}[Ideal Whitehead graphs $\mathcal{IW}(\vphi)$ and index list]
The \emph{local stable Whitehead graph} $\mathcal{SW}(g)$ is the subgraph of $\mathcal{LW}(g)$ obtained by restricting precisely to vertices with periodic direction labels and the edges of $\mathcal{LW}(g)$ connecting them. In this circumstance of a PNP-free train track representative on the rose, the \emph{ideal Whitehead graph $\mathcal{IW}(\vphi)$ of $\vphi$} is isomorphic to $\mathcal{SW}(g)$.

The \emph{index list} of $\vphi$ will be $\{1-\frac{k_1}{2}, \dots, 1-\frac{k_n}{2}\}$ where $k_i$ is the number of vertices in the $i^{th}$ component of $\mathcal{IW}(\vphi)$, where repetitions of index list values occur when multiple components have the same number of vertices.
\end{df}

We need one more notion of a Whitehead graph. 
\begin{df}[Limited Whitehead graphs $\mathcal{W}_L(g)$]
Let $g \colon \Gam \to \Gam$ be a graph map on the rose. The \emph{limited Whitehead graph} ($\mathcal{W}_L(g)$) for $g$ is the union of the set of turns taken by $g(e)$ for each $e \in \EG$.
\end{df}

\subsection{Ideal decompositions and cyclically admissible sequences}{\label{ss:id}}

Let $r\ge 2$ and continue to let $X$ be a free basis for $F_r$. By a \emph{standard Nielsen generator}, we will mean an automorphism $\vphi \in \autr$ such that there exist $x,y\in X^{\pm 1}$ with $\vphi(x)=yx$ and $\vphi(z)=z$ for each $z\in X^{\pm 1}$ with $z\ne x^{\pm 1}$. We specify such $\vphi$ using notation $\vphi=[x\mapsto yx]$.

In \cite{IWGII} it is proved that, if the ideal Whitehead graph of $\vphi \in \mathcal{A}_r$ is a connected $(2r-1)$-vertex graph, then there exists a rotationless power $\vphi^R$ with a PNP-free train track representative on the rose that is a composition of graph maps corresponding to standard Nielsen generators:

\begin{prop}[{\cite[Proposition~3.3]{IWGI}}]{\label{P:ID}}
Let $\vphi \in \mathcal{A}_r$ be such that $\mathcal{IW}(\vphi)$ is a connected $(2r-1)$-vertex graph.
Then there exists a rotationless power $\psi=\vphi^R$ and PNP-free train track representative $g$ of $\psi$ on the rose decomposing into the graph maps 
\begin{equation}{\label{e:id}}
\Gam=\Gamma_0 \xrightarrow{g_1} \Gamma_1 \xrightarrow{g_2} \cdots \xrightarrow{g_{n-1}} \Gamma_{n-1} \xrightarrow{g_n} \Gamma_n=\Gam
\end{equation}
\noindent where
{\begin{description}
\item [(I)] The index set $\{1, \dots, n \}$ is viewed as the set $\Z$/$n \Z$ with its natural cyclic ordering.
\newline
\item [(II)] Let $\Gamma_0$ be the base rose $R_r$ with the edges identified with the fixed generators of $F_r$. Each
$\Gamma_s$ is a rose with a marking $m_s \colon \Gamma_0 \to \Gamma_s$ so that, if we denote $m_s(e_t) = e_{s,t}$, for some $i_s$, $j_s$ with $e_{s,i_s} \neq (e_{s,j_s})^{\pm 1}$, we have 
\begin{equation}
g_s(e_{s-1,t}):=
      \begin{cases}
        e_{s,i_s} e_{s,j_s} \text{ for } t=j_s \\
        e_{s,t} \text{ for } e_t \neq e_{j_s}^{\pm 1}
      \end{cases}
\end{equation}
\item [(III)] For each $e_i \in \mathcal{E}(\Gamma)$ such that $i \neq j_n$, we have that $Dg(e_i) = e_i$
\item [(IV)] $m_n$ is the identity map and $e_{n,t} = e_t$ for each $1 \leq t \leq 2r$.
\end{description}}
\end{prop}

As in \cite{IWGII}, we call train track maps satisfying (I)-(IV) of the proposition \emph{ideally decomposable (i.d.)} and call the decomposition in (\ref{e:id}) an \emph{ideal decomposition (i.d.)}.

We let $u \colon \{1, \dots, n \} \to \{1, \dots, 2r \}$ and $a \colon \{1, \dots, n \} \to \{1, \dots, 2r \}$ be defined by $u(s) = j_s$ and $a(s) = i_s$. Notice that the direction $e_{s, j_s}$ is missing from the image of $dg_s$. We denote it by $d^u_s$, i.e. $d^u_s=e_{s,u(s)}$. The direction $e_{s, i_s}$ is the image of two directions under $dg_s$ and we denote it by $d^a_s$, i.e. $d^a_s=e_{s,a(s)}$.

We use the notation $f_k:= g_k \circ \cdots \circ g_1 \circ g_n \circ \cdots \circ g_{k+1}\colon \Gamma_k \to \Gamma_k$ and
%~\\
%\vspace{-3mm}
\[g_{k,i}:=
\begin{cases}
    g_k \circ \cdots \circ g_i\colon \Gamma_{i-1} \to \Gamma_k \; \text{if} \; k>i\\
    g_k \circ \cdots \circ g_1 \circ g_n \circ \cdots \circ g_i \; \text{if} \; k<i\\
     \text{the identity } id \; \text{if} \; k=i-1. \\[-2mm]
\end{cases}
\]

\begin{rk}{\label{R:Cyclic}}
It is proved in \cite{IWGI} that if $\Gamma = \Gamma_0 \xrightarrow{g_1} \Gamma_1 \xrightarrow{g_2} \cdots \xrightarrow{g_{n-1}}\Gamma_{n-1} \xrightarrow{g_n} \Gamma_n = \Gamma$ is an ideal decomposition of $g$, then $\Gamma_k \xrightarrow{g_{k+1}} \Gamma_{k+1} \xrightarrow{g_{k+2}} \cdots \xrightarrow{g_{k-1}} \Gamma_{k-1} \xrightarrow{g_k} \Gamma_k$ is an ideal decomposition of $f_k$ (and $f_k$ is in fact also a PNP-free train track representative of the same $\psi=\vphi^R$ as $g$).
\end{rk}

\cite{kp15} generalizes the notion of an ideal decomposition to a sequence of Nielsen generators (and hence also to their corresponding graph maps).

\begin{df}[(Cyclically) admissible sequences]{\label{D:Cyclic}}
We call a sequence of standard Nielsen automorphisms $\vphi_1,\dots, \vphi_n$ (with $\vphi_i=[x_i \mapsto y_i x_i]$) \emph{admissible} if each pair $(\vphi_i,\vphi_{i+1})$ is \emph{admissible}, i.e. either $x_{i+1} = x_i$ and either $y_{i+1} \neq y_i^{-1}$ or $y_{i+1} = x_i$ and $x_{i+1} \neq y_i^{-1}$. We then also call the sequence $g_{\vphi_1},\dots, g_{\vphi_n}$ \emph{admissible}. If the indices are viewed mod $n$, i.e $(\vphi_n,\vphi_1)$ is also admissible, we say that the sequences are \emph{cyclically admissible}.
\end{df}

\subsection{Lamination train track (ltt) structures}{\label{ss:ltt}}

Instead of just viewing the $g_i$ in an ideal decomposition as graph maps, we want to be able to ``blow up'' the vertices of the graphs $\Gamma_{i-1}$, $\Gamma_i$ to give them extra structure and then have $g_i$ induce a move between these structures. The structures are the ``ltt structures'' of \cite{IWGI} defined as follows (the moves are the extensions and switches of Subsection \ref{ss:se}):

\begin{df}[Lamination train track (ltt) structures $G(g)$]{\label{d:ltt}}

Let $g \colon \Gamma \to \Gamma$ be a PNP-free train track map on the rose (with vertex $v$). The \emph{colored local Whitehead graph} for $g$ at $v$, denoted $\cwg$, is the graph $\mathcal{LW}(g)$ with the subgraph $\mathcal{SW}(g)$ colored purple and $\mathcal{LW}(g)- \mathcal{SW}(g)$ colored red (nonperiodic direction vertices are red). Let $N(v)$ be a contractible open neighborhood of $v$ and $\Gamma_N=\Gamma-N(v)$. For each $E_i \in \mathcal{E}^+(\Gam)$, we add vertices labeled by the directions $E_i$ and $\overline{E_i}$ at the corresponding boundary points of the partial edge $E_i-(N(v) \cap E_i)$. The \emph{lamination train track (ltt) structure $G(g)$} for $g$ is formed from $\Gamma_N \bigsqcup \mathcal{CW}(g)$ by identifying each vertex $d_i$ in $\Gamma_N$ with the vertex $d_i$ in $\mathcal{CW}(g)$. Vertices for nonperiodic directions are red, edges of $\Gamma_N$ are black, and all periodic vertices are purple.

By the \emph{smooth structure} on $G(g)$ we mean the partition of the edges at each vertex into two sets: $\mathcal{E}_b$ (the black edges of $G(g)$) and $\mathcal{E}_c$ (the colored edges of $G(g)$). We call any path in $G(g)$ alternating between colored and black edges \emph{smooth}.
\end{df}

\begin{df}[Achieved ltt structure]{\label{d:altt}}
We call a $2r$-vertex ltt structure $G$ \emph{achieved} if $G = G(g)$ for some PNP-free rotationless i.d. train track representative of a $\phar$.
\end{df}

In light of Remark \ref{R:Cyclic}, for an ideal decomposition $\Gamma = \Gamma_0 \xrightarrow{g_1} \Gamma_1 \xrightarrow{g_2} \cdots \xrightarrow{g_{n-1}}\Gamma_{n-1} \xrightarrow{g_n} \Gamma_n = \Gamma$, one in fact has an associated cyclic sequence of ltt structures $G_1=G(f_1), \dots, G_n=G(f_n)=G(g)$. For this reason we sometimes write an ideally decomposed map as $(g_1, \dots, g_n; G_1, \dots, G_n)$. 

\begin{rk}{\label{r:UniqueRed}}
Notice that the red vertex of each $G_i$ is labeled by $d^u_i$, which is the unique nonperiodic direction. Also, since $g_i$ (viewed as an automorphism) replaces each copy of $x_{u(i)}$ with $x_{a(i)}x_{u(i)}$, we have that $G_i$ has a unique red edge and it connects $d^u_i$ to $\ol{d^a_i}$ (this argument is formulated more completely in \cite{IWGI}).
\end{rk}

The next subsection (Subsection \ref{ss:se}) is dedicated to describing the only two possible categories of moves relating the ltt structures $G_i$ and $G_{i+1}$, in an ideal decomposition.

\subsection{Switches and extensions}{\label{ss:se}}

We need an abstract notion of an ltt structure so that we can take a potential ideal Whitehead graph, extend it to an abstract ltt structure, and see if this structure could be the ltt structure for a representative, as in Proposition \ref{P:ID}. In fact, we will want an entire sequence of ltt structures (and moves between them) to have an ideal decomposition.

Since we frequently deal with graphs whose vertices are labeled by the directions at the vertex of a rose (sometimes abstractly), we first establish notation for discussing such graphs and their vertex-labeling sets. 

\begin{df}[Edge pair labelings]
We call a $2r$-element set of the form $\{x_1, \overline{x_1}, \dots, x_r, \overline{x_r} \}$, with elements paired into \emph{edge pairs} $\{x_i, \overline{x_i}\}$, a \emph{rank}-$r$ \emph{edge pair labeling set} (we write $\overline{\overline{x_i}}=x_i$). We call a graph with vertices labeled by an edge pair labeling set a \emph{pair-labeled} graph, and an \emph{indexed pair-labeled} graph if an indexing is prescribed.
\end{df}

\begin{df}{\label{d:abstractltt}}[(Abstract) lamination train track (ltt) structures]
A \emph{lamination train track (ltt) structure} is a colored pair-labeled graph $G$ (black edges are included, but not considered colored) satisfying
\noindent {\begin{description}
\item [I] Each vertex has valence $\geq 2$.
\item [II] Each edge has $2$ distinct vertices.
\item [III] Vertices are either purple or red.
\item [IV] Edges are of $3$ types:
{\begin{description}
\item[Black Edges] $G$ has a single black edge connecting each pair of (edge-pair)-labeled vertices. There are no other black edges. Thus, each vertex is contained in a unique black edge.
\item[Red Edges] A colored edge is red if and only if at least one of its endpoint vertices is red.
\item[Purple Edges] A colored edge is purple if and only if both endpoint vertices are purple.
\end{description}}
\item [V] Two distinct colored edges never connect the same pair of vertices.
\end{description}}
It is proved in \cite{IWGI} that, in our situation, $G(g)$ has a unique red vertex and red edge (see Remark \ref{r:UniqueRed}). Hence, we also require:
\noindent {\begin{description}
\item [VI] $G$ has precisely $2r-1$ purple vertices, a unique red vertex, and a unique red edge.
\end{description}}
\end{df}

We consider ltt structures \emph{equivalent} that differ by an ornamentation-preserving graph isomorphism (homeomorphism taking vertices to vertices and edges to edges and preserving colors and labels).

We say that an ltt structure is \emph{based at} a rose $\Gam$ if its vertex labelling set is $\mD(\Gam)$.

\begin{nt}{\label{n:ltt}}
We let $[x,y]$ denote the edge connecting the vertex pair labeled by $\{x,y\}$. To be consistent with the situation of Definition \ref{d:ltt}, we call the unique red vertex $d^u$. Also, we call the unique red edge $[d^u, \ol{d^a}]$. In particular, $\ol{d^a}$ labels the purple vertex of the red edge.

Suppose one has a sequence of ltt structures $G_k, \dots, G_n$ such that, for each $k \leq i \leq n$, we have that $G_i$ is based at the rose $\Gam_i$. Then the notation will carry indices (for example, $d^u_i$ will denote the red vertex in $G_i$). Additionally, $e_{u(i)}$ denotes the edge in the $\Gam_i$ whose initial direction is $d^u_i$ and $e_{a(i)}$ denotes the edge in $\Gam_i$ whose terminal direction is $\ol{d^a_i}$.

We denote the colored subgraph of $G$ by $C(G)$ and the purple subgraph by $\mathcal{P}(G)$. If $\mG \cong \mathcal{P}(G)$, we say $G$ is an \emph{ltt structure for $\mG$}. We call an ltt structure \emph{admissible} which is \emph{birecurrent}, i.e has a locally smoothly embedded line traversing each edge infinitely many times as $\mathbb{R}\to \infty$ and as $\mathbb{R}\to -\infty$.

Given a graph map $g_k \colon \Gam_{k-1} \to \Gam_k$ as in Proposition \ref{P:ID}(II) and ltt structures $G_i$ based at $\Gam_i$, for $i = k-1, k$, the induced map $D^Tg_k$ (when it exists) is defined by sending each vertex $d$ to $Dg_k(d)$ and each edge $[d_1, d_2]$ to $[Dg_k(d_1), Dg_k(d_2)]$.
\end{nt}

\begin{df}[Generating triples]{\label{d:GeneratingTriples}}
By a \emph{generating triple} $(g_k; G_{k-1}, G_k)$ for $\mG$ we mean an ordered set of three objects, where 
\begin{description}
\item [gtI] $g_k \colon \Gam_{k-1} \to \Gam_k$ is a graph map as in Proposition \ref{P:ID}(II).
\item [gtII] $G_i$, for $i = k-1, k$, is an ltt structure for $\mG$ based at $\Gam_i$ and in fact:
{\begin{itemize}
\item the red vertex of $G_k$ is labelled by $e_{u(k)}$, i.e. $d^u_k = e_{u(k)}$, and
\item the red edge of $G_k$ is $[e_{u(k)}, \ol{e_{a(k)}}]$, i.e. $d^a_k = e_{a(k)}$.
\end{itemize}}
\item [gtIII] The induced map of ltt structures $D^Tg_k \colon G_{k-1} \to G_k$ exists and restricts to a graph isomorphism from $\mP(G_{k-1})$ to $\mP(G_k)$.
\end{description}
\end{df}

The triple is \emph{admissible} if
\begin{description}
\item[1] either $u(k-1) = u(k)$ or $u(k-1) = a(k)$ and
\item[2] each $G_i$ is admissible, i.e. birecurrent.
\end{description}

\vskip5pt

Given an ltt structure $G_k$ for $\mG$ and a \emph{determining} purple edge $[d^a_k,d_{k,l}]$, there are potentially two admissible triples $(g_k; G_{k-1}, G_k)$ that could arise in an ideal decomposition. They correspond to the situations where $u(k-1) = u(k)$ (the ``extension'' situation) and the situations where $u(k-1) = a(k)$ (the ``switch'' situation).

\begin{df}[Extensions]
The \emph{extension} determined by $[d^a_k,d_{k,l}]$ is the generating triple $(g_k; G_{k-1}, G_k)$ for $\mG$ satisfying
\begin{description}
\item [extI] $u(k-1) = u(k)$, thus the restriction of $D^Tg_k$ to $\mP(G_{k-1})$ is an isomorphism sending the vertex $e_{k-1,i}$ to the vertex $e_{k,i}$ for each $i \neq u(k-1)$.
\item [extII] $e_{k-1, u(k-1)} = d^u_{k-1}$, i.e. is the red vertex of $G_{k-1}$.
\item [extIII] $\ol{d^a_{k-1}}=d_{k-1,l}$, i.e. the purple vertex of the red edge in $G_{k-1}$ is $d_{k-1,l}$.
\end{description}
\end{df} 

\begin{df}[Switches]
The \emph{switch} determined by $[d^a_k,d_{k,l}]$ is the generating triple $(g_k; G_{k-1}, G_k)$ for $\mG$ satisfying
\begin{description}
\item [swI] $u(k-1) = a(k)$, thus $D^Tg_k$ restricts to an isomorphism from $\mP(G_{k-1})$ to $\mP(G_{k})$ defined by
\[
\begin{cases}
    e_{k-1, u(k)} \mapsto e_{k, a(k)} = e_{k, u(k-1)}\\
    e_{k-1, s} \mapsto e_{k, s} \text{ for } s \neq u(k)\\
\end{cases}
\]
\item [swII] $e_{k-1, a(k)} = d^u_{k-1}$, i.e. is the red vertex of $G_{k-1}$.
\item [swIII] $\ol{d^a_{k-1}}=d_{k-1,l}$, i.e. the purple vertex of the red edge in $G_{k-1}$ is $d_{k-1,l}$.
\end{description}
\end{df} 

\begin{df}[Admissible compositions]
An {admissible composition} $(g_{i-k}, \dots, g_i; G_{i-k-1}, \dots, G_i)$ for $\mG$ with $0 \leq k <i$ consists of
\begin{itemize}
\item a sequence of automorphisms $g_{i-k}, \dots, g_i$ such that $\Gamma_{i-k-1} \xrightarrow{g_{i-k}} \cdots  \xrightarrow{g_i} \Gamma_i$ satisfies Proposition \ref{P:ID} (I)-(III) and
\item a sequence of ltt structures $G_{i-k-1}, \dots, G_i$ for $\mG$ so that each $(g_j; G_{j-1}, G_j)$, with $i-k \leq j \leq i$ is either an admissible switch or an admissible extension.
\end{itemize}
When $k=1$, we call the composition an \emph{i.d. admissible triple}. We may also write that $(g_{i-k-1}, \dots, g_i; G_{i-k-1}, \dots, G_i)$ is a admissible composition if 
$$g_{i-k-1}=[e_{i-k-2,u(i-k-1)} \mapsto e_{i-k-1,a(i-k-1)}e_{i-k-1,u(i-k-1)}] \colon \Gam_{i-k-2} \to \Gam_{i-k-1}$$ 
is additionally as in Proposition \ref{P:ID}.

Two admissible compositions $(g_{i-k}, \dots, g_i; G_{i-k-1}, \dots, G_i)$ and $(g'_{i-k}, \dots, g'_i; G'_{i-k-1}, \dots, G'_i)$ for the same $\mG$ are considered \emph{equivalent} if, for each $i-k-1 \leq j \leq i$, we have that $G_j$ is equivalent to $G_j'$ and that $g_j$ and $g_j'$ correspond to the same standard Nielsen generator.
\end{df}

\begin{df}[Extended admissible composition]{\label{d:extended}}
Any admissible composition in rank $r$ can be \emph{extended} to an admissible composition in each rank $r' > r$, where the automorphisms are extended to be the identity on $X_{r+1}, \dots, X_{r'}$ and the ltt structures are extended by adding purple vertices for $X_{r+1}, \overline{X_{r+1}}, \dots, X_{r'}, \overline{X_{r'}}$ and the black edges between each pair $X_i, \overline{X_i}$.
\end{df} 

\begin{rk}
The automorphisms of an (extended) admissible composition will form an admissible sequence in the sense of Definition \ref{D:Cyclic}.
\end{rk}

\subsection{Ideal decomposition diagrams ($\mathcal{ID}(\mathcal{G})$)}{\label{ss:idd}}

Recall that a directed graph is strongly connected if for each pair of vertices $v_1$, $v_2$ in the graph, the graph contains a directed path from $v_1$ to $v_2$. Notice that one can find the union of the strongly connected components in a graph by taking the union of all of the directed loops in the graph. Since an $\mathcal{ID}$ diagram should contain a loop for each ideally decomposed PNP-free representative with a given ideal Whitehead graph, we define, and give a procedure for constructing, such a diagram in \cite{p12a}, where they are called ``$\mathcal{AM}$ diagrams.''

\begin{df}[Ideal decomposition diagrams] Let $\mG$ be a connected $(2r-1)$-vertex graph. The \emph{ideal decomposition diagram} for $\mG$ (or $\idg$) is defined to be the disjoint union of the maximal strongly connected subgraphs of the directed graph where:
\begin{description}
\item[Nodes] The nodes are equivalence classes of admissible indexed ltt structures for $\mG$.
\item[Edges] For each equivalence class of an i.d. admissible triple $(g_i; G_{i-1}, G_i)$ for $\mG$, there is a directed edge $E(g_i; G_{i-1}, G_i)$ in $\idg$ from the node $[G_{i-1}]$ to the node $[G_i]$.
\end{description}
\end{df}

Inspired by the Full Irreducibility Criterion of \cite{IWGII}, the following lemma (\cite[Lemma 4.2]{IWGII}) tells us when a loop in $\mathcal{ID}(\mG)$ defines a train track representative of a fully irreducible $\vphi$ with $\mathcal{IW}(\vphi) \cong \mG$:

\begin{lem}[\cite{IWGII} Lemma 4.2]{\label{l:RepresentativeLoops}}
Suppose $\mG$ is a connected $(2r-1)$-vertex graph, $g = g_n \circ \cdots \circ g_1$ is rotationless, and 
$$L(g_1, \dots, g_n; G_0, G_1 \dots, G_{n-1}, G_n)= E(g_1; G_{0}, G_1) * \dots * E(g_n; G_{n-1}, G_n)$$ 
is a loop in $\mathcal{ID}(\mG)$ satisfying each of the following:
\begin{description}
\item [A] The purple edges of $G(g_{n,1})$ correspond to turns  $\{d_1, d_2\}$ taken by some $g^p(e_j)$, where $p \geq 1$, $e_j \in \mathcal{E}(\Gamma_0)$, and $d_1$ and $d_2$ are periodic directions for $g$.
\item [B] For each $1 \leq i,j \leq r$, there exists some $p \geq 1$ such that $g^p(E_j)$ contains either $E_i$ or $\overline{E_i}$.
\item [C] The map $g$ has no periodic Nielsen paths. \\[-5mm]
\end{description}
\noindent Then $g$ is a train track representative of some $\vphi \in \mathcal{A}_r$ such that $\iw(\vphi)=\mG$. 
\end{lem}

\begin{rk}
We remark, leaving the proof to the reader, that Lemma \ref{l:RepresentativeLoops} still holds for cyclically admissible sequences of generators $g_1, \dots, g_n$ when $\mP(G(g_{n,1})) \cong \mG$.
\end{rk}

We will use in Section \ref{s:HigherRanks} the following corollary indirectly proved in \cite{IWGI}:

\begin{cor}{\label{c:RepresentativeLoops}}
Under the conditions of Lemma \ref{l:RepresentativeLoops}, $G(g) = G_0$
\end{cor}

\begin{proof} By extI and swII, $Dg$ induces an isomorphism of $\mP(G_0)$ that, if $g$ is rotationless, exactly fixes each vertex of $\mP(G_0)$. In particular, each purple vertex of $G_0$ is also purple in $G(g)$. Since $G_0$ has $2r-1$ purple vertices and $g$ cannot have more than $2r-1$ periodic directions, $G_0$ and $G(g)$ have the same purple vertices. Lemma \ref{l:RepresentativeLoops}A then implies that $\mP(G_0)$ is a subgraph of $\mP(G(g))$. But, since $g$ is a PNP-free train track map on the rose, $\mP(G(g)) \cong \swg \cong \mG \cong \mP(G_0)$. So $G_0$ and $G(g)$ have the same purple subgraph and we only need that they have the same red edge. However, the red edge of both $G_0$ and $G(g)$ is determined by the final generator $g_n$ in the decomposition of $g$.
\qedhere
\end{proof}

\vskip10pt

The following lemma is primarily proved in \cite{kp15}:

\begin{lem}{\label{L:LimitedWGs}}
Suppose $g_m, \dots, g_n$ is an admissible sequence of Nielsen generators. Then 
\begin{description}
%~\\
%\vspace{-6.23mm}
\item [A] the map $g$ is a graph map,
\item [B] for each $E \in \mE^+(\Gam)$, we have that $g(E)$ contains $E$,
\item [C] For any $1\le m < n$ we have
$$\mW_L(g_{n,m})= [\mathcal T(g_{n})] \cup [Dg_{n}(\mathcal T(g_{n-1,m}))],$$
where for $m=1$ we interpret both $g_{m-1}=g_0$ and $g_{n-1,n-1}$ as the identity map of $R_r$.
\item [D] and for each $m \leq s \leq n$, $$\mW_L(g_{n,m}) = Dg_{n,s+1}(\mW_L(g_{s,m})) \cup \mW_L(g_{n,s+1}).$$
\end{description}
\end{lem}

\begin{proof} (A) follows from \cite[Lemma 3.10(2)]{kp15}.

We prove (B) by induction on $n$. Suppose it were true for $n-1$, i.e. that for each $i$, we have that $g_{n-1,m}(E_i)$ contains $E_i$. Notice that $g_{n,m}(E_i)=g_n(e_1) \cdots g_n(e_k)$, where $g_{n-1,m}(E_i)=e_1 \cdots e_k$, and notice that $g_n(E_i)$ contains $E_i$. Since $g_{n,m}$ is a graph map, so that the image of each edge has no cancellation, this implies $g_{n,m}(E_i)$ contains $E_i$.

(C) follows from \cite[Lemma 3.7]{kp15} and (D) is implied by (C).
\qedhere
\end{proof}

%\cite{IWGI} and more directly proved in \cite{IWGIII}.

%\begin{lem}{\label{L:LimitedWGs}}
%Suppose $(g_m, \dots, g_n; G_m, \dots, G_n)$ is an admissible composition with the standard notation that $g_i =  [e_{i-1,u(i)} \to e_{i,a(i)} e_{i,u(i)}]$, for each $m \leq i \leq n$. Then 
%\begin{description}
%~\\
%\vspace{-6.23mm}
%\item [A] for each $E_{m-1,t} \in \mE^+(\Gam_{m-1})$, we have that $g(E_{m-1,t})$ contains $E_{n,t}$,
%\item [B] the map $g$ is a graph map,
%\item [C] $$\mW_L(g_{n,m}) = \bigcup_{k = m}^{n}[Dg_{n,k+1}(\{\ol{e_{k,a(k)}}, e_{k,u(k)})],$$
%\item [D] and for each $m \leq s \leq n$, $$\mW_L(g_{n,m}) = Dg_{n,s+1}(\mW_L(g_{s,m})) \cup \mW_L(g_{n,s+1}).$$
%\end{description}
%\end{lem}

%The following lemma is indirectly proved in \cite{IWGI} and more directly proved in \cite{IWGIII}.

%\begin{lem}{\label{L:LimitedWGs}}
%Suppose $(g_m, \dots, g_n; G_m, \dots, G_n)$ is an admissible composition with the standard notation that $g_i =  [e_{i-1,u(i)} \to e_{i,a(i)} e_{i,u(i)}]$, for each $m \leq i \leq n$. Then 
%\begin{description}
%~\\
%\vspace{-6.23mm}
%\item [A] for each $E_{m-1,t} \in \mE^+(\Gam_{m-1})$, we have that $g(E_{m-1,t})$ contains $E_{n,t}$,
%\item [B] the map $g$ is a graph map,
%\item [C] $$\mW_L(g_{n,m}) = \bigcup_{k = m}^{n}[Dg_{n,k+1}(\{\ol{e_{k,a(k)}}, e_{k,u(k)})],$$
%\item [D] and for each $m \leq s \leq n$, $$\mW_L(g_{n,m}) = Dg_{n,s+1}(\mW_L(g_{s,m})) \cup \mW_L(g_{n,s+1}).$$
%\end{description}
%\end{lem}

\subsection{Nielsen path prevention sequences}{\label{ss:iNPprevention}}

\cite[Section 5]{IWGII} contains a procedure for identifying PNPs, which can also be used to show when they do not exist. The idea is the following. Suppose that $g=g_n \circ \cdots \circ g_1 \colon \Gamma \to \Gamma$ is an expanding irreducible i.d. train track map. We can assume further that $g$ is rotationless so that any PNP $\rho$ is in fact an NP. By \cite{bh92}, $\rho$ decomposes into iNPs. So we can assume that $\rho$ is an iNP and, in particular, can (by \cite[Lemma 3.4]{bh92}, written as Lemma \ref{l:iNP} above) be written as $\rho = \ol{\rho_1}\rho_2$, where $\rho_1$ and $\rho_2$ are nontrivial legal paths and the turn between $\rho_1$ and $\rho_2$ is a nondegenerate illegal turn based at the vertex $v$ of the rose $\Gamma$. Notice that, while the endpoints need to be fixed points, they do not necessarily need to be vertices. If $\rho$ is indeed an iNP, then $g(\rho) \cong \rho$ rel endpoints and, in fact, $g(\rho_1) = \tau\rho_1$ and $g(\rho_2) = \tau\rho_2$ for some common legal path $\tau$. For this to happen, it would be necessary that, for each $1 \leq k \leq n$, we have that $g_{k,1}(\rho_1) = \tau_k\rho_1'$ and $g_{k,1}(\rho_2) = \tau_k\rho_2'$, where $\rho_1'$, $\rho_2'$, and $\tau_k$ are each legal paths for $g_{n,k+1}$, $\rho_1'$ and $\rho_2'$ are nontrivial, and the turn between $\rho_1'$ and $\rho_2'$ is the illegal turn for $g_{n,k+1}$, i.e. the illegal turn for $g_{k+1}$. Thus, one can start with the unique illegal turn for $g$ (i.e. $\{e_u,e_a\}$ where $g_1=[e_u \mapsto e_ae_u]$) and systematically check whether any $\rho_1$ starting with $e_u$ and $\rho_2$ starting with $e_a$ satisfy all of the above. 

For some $g$ as above, for any $\rho_1$ and $\rho_2$ one tries to build, one has that, for some $k < n$, the turn between $\rho_1'$ and $\rho_2'$ is legal for $g_{k+1}$. This is the case for any expanding irreducible i.d. train track map $g$ starting with the composition of \cite[Lemma 5.5]{IWGII}, prompting the introduction of the definition of a ``Nielsen path prevention sequence'' \cite{IWGII}:

\begin{df}[Nielsen path prevention sequence]
A \emph{Nielsen path prevention sequence} is an admissible composition $(g_1, \dots, g_k; G_0, \dots, G_k)$ so that if $(g_1', \dots, g_n'; G_0', \dots, G_n')$ is an i.d., where $g=g_n' \circ \cdots \circ g_1'$ is a rotationless expanding irreducible train track map, $n \leq k$, and $(g_1, \dots, g_k; G_0, \dots, G_k) \sim (g_1', \dots, g_k'; G_0', \dots, G_k')$, then $g$ has no PNPs.
\end{df}

%We will use in Section \ref{s:CutVertices} the following Nielsen path prevention sequence.

The proof of Lemma \ref{l:SpecificLegalizingNPseq} below (and also that of \cite{kp15}), prompts a revision of the definition of a Nielsen path prevention sequence so that the procedure applied (introduced in \cite[Proposition 5.2]{IWGII} ends in each case in the legal turn scenario of \cite{IWGII} Proposition 5.2(IIc) and so that taking powers is never required in the procedure. Finally, it only relies on a sequence being admissible in the sense of \cite{kp15}, and Definition \ref{D:Cyclic} above. We call such a sequence a \emph{legalizing} Nielsen path prevention sequence.

\begin{lem}{\label{l:LegalizingNPseq}} 
A legalizing Nielsen path prevention sequence in rank $r$ is a Nielsen path prevention sequence in each rank $r' > r$ where the extension is by the identity on $X_{r+1}, \dots, X_{r'}$.
\end{lem}

\begin{proof}
This follows from the proof of \cite[Proposition 5.2]{IWGII}.
\end{proof}

\begin{lem}{\label{l:SpecificLegalizingNPseq}} 
Let $g_{\vphi}$ be defined as
$$
\vphi =
\begin{cases} a \mapsto ac \bar{b} ca \bar{b} cacac \bar{b} ca  \\
b \mapsto \bar{a} \bar{c} b \bar{c} \bar{a} \bar{c} \bar{a} \bar{c} b  \\
c \mapsto cac \bar{b} ca \bar{b} cac
\end{cases}
$$
with the decomposition:
$$
g_1 = [a \mapsto a \bar{b}],
g_2 = [b \mapsto \bar{a}b],
g_3 = [c \mapsto c \bar{b}],
g_4 = [c \mapsto ca],
g_5 = [b \mapsto \bar{c}b],
$$
$$ 
g_6 = [a \mapsto a \bar{b}],
g_7 = [a \mapsto ac],
g_8 = [b \mapsto \bar{a}b],
g_9 = [b \mapsto \bar{c}b]
$$ 
Then $g_{\vphi}^2$ is a legalizing Nielsen path prevention sequence.
\end{lem}

\begin{proof} The verification process is long and so we just show one case. The other cases are similar, and in particular proceed as in \cite[Section 5]{IWGII}. 

Suppose $\rho$ were an iNP for some $g_{\vphi}^p$ Then $\rho$ would have to contain the illegal turn $\{\bar{a},b\}$ for $g_{\vphi}^p$ and (possibly after reversing its orientation) could be written $\overline{\rho_1}\rho_2$ where $\rho_1=be_2'e_3'\dots$ and $\rho_2=\bar{a}e_2e_3\dots$ are legal. Now, $g_1(b)=b$ and $g_1(\bar{a})=b\bar{a}$. So $\rho_1$ would contain an additional edge $e_2$. Also, (since the illegal turn for $g_2$ is $\{b,\bar{a}\}$ and $\bar{a}$ is not in the image of $Dg_1$) we have that $Dg_1(e_2)=b$. So either $e_2=b$ or $e_2=\bar{a}$. We analyze here the case where $e_2=\bar{a}$ and leave the case of $e_2=b$ to the reader.

Since $g_{2,1}(b\bar{a})=\bar{a}b\bar{a}b\bar{a}$ and $g_{2,1}(\bar{a})=\bar{a}b\bar{a}$, we know $\rho_2$ would contain an additional edge $e_2'$ with $Dg_{2,1}(e_2')=\bar c$ (since $\{b,\bar{c}\}$ is the illegal turn for $g_3$ and $b$ is not in the image of $Dg_{2,1}$). The only option is $e_2'=\bar{c}$. Now, $g_{3,1}(b\bar{a})=\bar{a}b\bar{a}b\bar{a}$ and $g_{3,1}(\bar{a}\bar c)=\bar{a}b\bar{a}b\bar{c}$. After cancellation, we are left with the turn $\{\bar{a},\bar{c}\}$, which is illegal for $g_4$ and so we can proceed by applying $g_4$. Since $g_{4,1}(b\bar{a})=\bar{a}b\bar{a}b\bar{a}$ and $g_{4,1}(\bar{a}\bar c)=\bar{a}b\bar{a}b\bar{a}\bar{c}$, we must add an additional edge $e_3$ to $\rho_1$ satisfying that $Dg_{4,1}(e_3)=b$ (since the illegal turn for $g_5$ is $\{b,\bar{c}\}$ and $\bar{c}$ is not in the image of $Dg_{4,1}$). So $e_3=\bar{c}$. Since $g_{5,1}(b\bar{a}\bar{c}) = \bar{a}\bar{c}b\bar{a}\bar{c}b\bar{a}\bar{c}b\bar{a}\bar{c}$ and $g_{5,1}(\bar{a}\bar c)=\bar{a}\bar{c}b\bar{a}\bar{c}b\bar{a}\bar{c}$, we must add an additional edge $e_3'$ to $\rho_2$ satisfying that $Dg_{5,1}(e_3')=\bar{a}$ (since the illegal turn for $g_6$ is $\{b,\bar{a}\}$ and $b$ is not in the image of $Dg_{5,1}$). So either $e_3'=\bar{a}$ or $e_3'=b$. We analyze here the case where $e_3'=b$ and leave the case of $e_3'=\bar{a}$ to the reader.
 
Since $g_{6,1}(b\bar{a}\bar{c}) = b\bar{a}\bar{c}bb\bar{a}\bar{c}bb\bar{a}\bar{c}bb\bar{a}\bar{c}$ and $g_{6,1}(\bar{a}\bar{c}b)=b\bar{a}\bar{c}bb\bar{a}\bar{c}bb\bar{a}\bar{c}b\bar{a}\bar{c}b$, cancellation ends with the turn $\{b,\bar{a}\}$. This is not the illegal turn for $g_7$. So we could not have $\rho_1=b\bar{a}\bar{c}\dots$ and $\rho_2=\bar{a}\bar{c}b\dots$

Since the remaining cases yield a similar situation, $g_{\vphi}^2$ is a legalizing Nielsen path prevention sequence. 
 \qedhere
\end{proof}

\subsection{The axis bundle for a nongeometric fully irreducible outer automorphism}{\label{ss:AxisBundle}}

For each $r \geq 2$, we let $CV_r$ denote the Culler-Vogtmann Outer Space in rank $r$, as define in \cite{cv86}. Then $Out(F_r)$ is the isometry group of $CV_r$ (see \cite{fm12} \cite{a12}), where each $\vphi \in Out(F_r)$ acts on $CV_r$ from the right by precomposing the marking with an automorphism representing $\vphi$.

\begin{df}[Attracting tree $T_+^{\varphi}$]{\label{d:AttractingTree}}
Let $\vphi \in Out(F_r)$ be a fully irreducible outer automorphism. Then the action of $\vphi$ on $CV_r$ extends to an action on $\overline{CV_r}$ with North-South dynamics (see \cite{ll03}). We denote by $T_+^{\varphi}$ the attracting point of this action and by $T_-^{\varphi}$ the repelling point of this action.
\end{df}

\begin{df}[Fold lines]
A \emph{fold line} in $CV_r$ is a continuous, injective, proper function $\mathbb{R} \to CV_r$ defined by \newline
\noindent 1. a continuous 1-parameter family of marked graphs $t \to \Gamma_t$ and \newline
\noindent 2. a family of homotopy equivalences $h_{ts} \colon \Gamma_s \to \Gamma_t$ defined for $s \leq t \in \mathbb{R}$, each marking-preserving, \newline
\indent satisfying:
~\\
\vspace{-\baselineskip}
\begin{description}
\item [\emph{Train track property}] $h_{ts}$ is a local isometry on each edge for all $s \leq t \in \mathbb{R}$. %\\[-6mm]
\item [\emph{Semiflow property}] $h_{ut} \circ h_{ts} = h_{us}$ for all $s \leq t \leq u \in \mathbb{R}$ and $h_{ss} \colon \Gamma_s \to \Gamma_s$ is the identity for all $s \in \mathbb{R}$.
\end{description}
\end{df}

\begin{df}[Axis bundle $\mathcal{A}_{\varphi}$] 
$\mathcal{A}_{\varphi}$ is the union of the images of all fold lines $\mathcal{F} \colon \mathbb{R} \to CV_r$ such that $\mathcal{F}$(t) converges in $\overline{CV_r}$ to $T_{-}^{\varphi}$ as $t \to -\infty$ and to $T_{+}^{\varphi}$ as $t \to +\infty$. We call the fold lines the \emph{axes} of the axis bundle.
\end{df}

The following is a direct consequence of \cite[Theorem 4.6]{mp13}.

\begin{prop}
Suppose that $\vphi \in Out(F_r)$ is an ageometric fully irreducible outer automorphism such that $\mathcal{IW}(\vphi)$ is a connected $(2r-1)$-vertex graph. Then the axis bundle $\mathcal{A}_{\varphi}$ consists of a single unique axis if and only if $\mathcal{IW}(\vphi)$ has no cut vertex.
\end{prop}

\section{Constructing ideal Whitehead graphs in higher ranks}{\label{s:HigherRanks}}

In this section we describe a tool for using \cite[Theorem A]{IWGIII} to obtain classes of connected $(2r-1)$-vertex graphs in higher ranks. In particular, we prove the existence in each rank of an ideal Whitehead graph with a cut vertex (which then, by \cite[Theorem 3.9]{mp13}, gives an example in each rank of an ageometric fully irreducible whose axis bundle contains more than one axis).

\begin{prop}{\label{L:AchievedNodes}}
Let $g=g_n \circ \cdots \circ g_1 \colon \Gamma \to \Gamma$ be an i.d. train track representative of $\phar$ and $\mG \cong \iwp$. Suppose that $g$ is a Nielsen path prevention sequence. Let $G'$ be an ltt structure labelling a node of the component of $\idg$ in which $G(g)$ labels a node. Then $G'$ is achieved.
\end{prop}

For the proof we will use the following lemma from \cite{IWGI}

\begin{lem}[\cite{IWGI} Lemma 3.1]{\label{L:pNpFreePreserved}} Let $g \colon \Gamma \to \Gamma$ be a PNP-free train track representative of $\vphi \in \mathcal{A}_r$ and
 $$\Gamma = \Gamma_0 \xrightarrow{g_1} \Gamma_1 \xrightarrow{g_2} \cdots \xrightarrow{g_{n-1}} \Gamma_{n-1} \xrightarrow{g_n} \Gamma_n = \Gamma$$
 a decomposition of $g$ into homotopy equivalences of marked graphs. Then the composition
 $$h \colon \Gamma_k \xrightarrow{g_{k+1}} \Gamma_{k+1} \xrightarrow{g_{k+2}} \cdots \xrightarrow{g_{k-1}} \Gamma_{k-1} \xrightarrow{g_k} \Gamma_k$$
 is also a PNP-free train track representative of $\vphi$ (in particular, $\mathcal{IW}(h) \cong \mathcal{IW}(g)$).
\end{lem}

\begin{proof}[Proof of Proposition \ref{L:AchievedNodes}]
By taking a power we may assume that $g$ is rotationless, that for each pair of edges $E_i,E_j \in \mE^+(\Gam)$, we have $g(E_j)$ contains either $E_i$ or $\overline{E_i}$, and that $\mW_L(g)=\mathcal{LW}(g)$ (so that, in particular, each purple edge of $G = G(g)$ is labelled by a turn taken by some $g(E_k)$ with $E_k \in \mE^+(\Gam)$). Let $L$ be the loop in $\idg$ originating at $G$ and realizing the ideal decomposition $g=g_n \circ \cdots \circ g_1$. Given another node $G'$ in the same component of $\idg$, strong connectivity implies that there exists a loop $L'$ originating at $G$ and containing $G'$. Let $g'=g'_{n'} \circ \cdots \circ g_1' \colon \Gam \to \Gam$ denote the i.d. train track map corresponding to this loop. By lengthening the loop $L'$ to also contain $L$ and then by taking a power if necessary, we may assume both that $g'$ is rotationless and that, for each pair of edges $E_i,E_j \in \mE^+(\Gam)$, we have that $g(E_j)$ contains either $E_i$ or $\overline{E_i}$. 

We first prove (using Lemma \ref{l:RepresentativeLoops} and Corollary \ref{c:RepresentativeLoops}) that $g' \circ g$ is a train track representative of some $\phar$ such that $\iwp \cong \mG$ and, in particular, that $G(g' \circ g) = G(g)$. Since $g'$ is rotationless, by Lemma \ref{L:LimitedWGs}D, each turn represented by a purple edge in $G(g)$ is additionally taken by some $g' \circ g(E_k)$ with $E_k \in \mE^+(\Gam)$, giving us that $g' \circ g$ satisfies the conditions of Lemma \ref{l:RepresentativeLoops}A. 
%There will be no additional turns taken since $L'$ is a loop based at $G = G(g)$. 
We already know that, for each $E_i, E_j \in \mE^+(\Gam)$, $g(E_j)$ contains either $E_i$ or $\overline{E_i}$. Additionally, by Lemma \ref{L:LimitedWGs}B, each $g'(E_i)$ contains $E_i$ (and $g'(\overline{E_i})$ contains $\overline{E_i}$). So $g' \circ g(E_j)$ contains either $E_i$ or $\overline{E_i}$, giving us that $g' \circ g$ satisfies Lemma \ref{l:RepresentativeLoops}B (and also that $g' \circ g$ is expanding and irreducible). Lemma \ref{l:RepresentativeLoops}C is given for $g' \circ g$ by the fact that $g$ is a Nielsen path prevention sequence.  Thus, Lemma \ref{l:RepresentativeLoops} implies that $g' \circ g$ is a train track representative of some $\phar$ such that $\iwp \cong \mG$ and Corollary \ref{c:RepresentativeLoops} implies $G(g' \circ g) = G(g)$.

Let $1 \leq m \leq n'$ be such that $G' = G'_m$ where
$$L' = L(g_1', \dots, g_n'; G_0' = G(g), G_1' \dots, G_m' \dots, G_{n-1}', G_n' = G(g)).$$ 
By Lemma \ref{L:pNpFreePreserved}, $h = g'_{m,1} \circ g \circ g'_{n,m+1}$ is also a PNP-free train track representative of $\vphi$. By Lemma \ref{l:RepresentativeLoops} and Corollary \ref{c:RepresentativeLoops}, for $G'$ to be achieved, it suffices to show that some power of $h$ satisfies the conditions of Lemma \ref{l:RepresentativeLoops}. Consider a rotationless power $h^R$, where $R \geq 2$. By the previous paragraph, we know that each purple edge of $G_0 = G(g' \circ g)$ is labelled by a turn taken by $g' \circ g(E_k)$ with $E_k \in \mE^+(\Gam)$. By Lemma \ref{L:LimitedWGs}A, we know that each $E_k$ is in the image of $g'_{n,m+1}$ and hence each such turn is also taken by the image of an edge under $(g' \circ g)^{R-1} \circ g'_{n,m+1}$. By Lemma \ref{L:LimitedWGs}D, $\mW_L(h^R)$ will contain the image under $Dg'_{m,1}$ of $\mP(G_0)$. But extI and swI imply that $Dg'_{m,1}$ acts as an isomorphism on $\mP(G_0)$, so each purple edge in $G'$ is a turn taken by the image of an edge under $h^r$. 
By Lemma \ref{L:LimitedWGs}, given $E_{m',i}$, $E_{m',j} \in \mE(\Gam_m')$, we have that $g'_{n,m+1}(E_{m',j})$ contains $E_{0,j}$. By the previous paragraph, $(g' \circ g)^{R-1}(E_{0,j})$ contains either $E_{0,i}$ or $\ol{E_{0,i}}$. Thus, by Lemma \ref{L:LimitedWGs}, we have that $h^R(E_{m',j}) = g'_{m,1} \circ g \circ g'_{n,m+1}(E_{m',j})$ contains either $E_{m',i}$ or $\ol{E_{m',i}}$. So Lemma \ref{l:RepresentativeLoops}B is satisfied by $g' \circ g$. Lemma \ref{l:RepresentativeLoops}C follows from Lemma \ref{L:pNpFreePreserved}.  
 \qedhere
\end{proof}

\begin{mainthmB}{\label{T:MainTheorem2}} Let $G_1$ and $G_2$ be two ltt structures achieved by expanding irreducible train track maps that are in fact legalizing Nielsen path prevention sequences and that are index-labeled in such a way so that the red edge of each is $[X_1, X_2]$. Suppose further that $G_1$ has $2r'$ vertices and $G_2$ has $2r''$ vertices. Let $\mI = \{X_{i_1}, \ol{X_{i_1}}, \dots, X_{i_k}, \ol{X_{i_k}}\}$ be some subset of the indexing set including $X_1$ and $X_2$ (hence $\ol{X_1}$ and $\ol{X_2}$). Suppose $\mG$ is a graph obtained from $C(G_1) \sqcup C(G_2)$ by gluing the vertices indexed by $\mI$ in $C(G_1)$ to the vertices indexed by $\mI$ in $C(G_2)$ and removing the red edge. Then $\mG \cong \iwp$ for some $\vphi \in \mA_{r}$ where $r=r'+r''-k$. \end{mainthmB}

\begin{proof} 
We let $\mI^+$ denote the subset $\{X_{i_1}, \dots, X_{i_k}\}$ of $\mI$.

Index the vertices in $\mG$ coming from $G_1$ as they are indexed in $G_1$. Reindex the remaining vertices of $\mG$ by $\{X_{r'+1}, \ol{X_{r'+1}}, \dots, X_r, \ol{X_r}\}$ in a manner preserving edge-pairing of $G_2$ (notice that these vertices of $\mG$ come from $C(G_2)$). Use this reindexing of $\mG$ to reindex the vertices of $C(G_2)$. Suppose that, with this new induced indexing on $G_1$ and $G_2$, we have that $G_1=G(g)$ and $G_2=G(h)$, by which we mean that there exist expanding irreducible train track maps $g \colon \Gam^1 \to \Gam^1$, $h \colon \Gam^2 \to \Gam^2$ of indexed roses where the edges of the $\Gam^i$, for $i=1,2$, are respectively indexed as in the induced indexing set of the $G_i$, each direction $E_j$ labels the same vertex as $X_j$, and each direction $\ol{E_{j}}$ labels the same vertex as $\ol{X_{j}}$.

By taking a power if necessary, we can assume that $g$ and $h$ are rotationless, that $g$ and $h$ are strictly irreducible, that $\mW_L(g)=\mathcal{LW}(g)$,  and that $\mW_L(h)=\mathcal{LW}(h)$.

By considering $\{X_{1}, \dots, X_r\}$ as a free basis for the free group $F_r$, one can use it to give a marking $m$ of a directed $r$-petaled rose where each distinct positively oriented edge is sent to a distinct basis element. We call this marked graph $\Gam$ and, for each $1 \leq i \leq r$, let $E_i$ denote the edge of $\Gam$ such that $m(E_i)=X_i$. We can then view $g$ and $h$ as (no longer irreducible) train track maps of $\Gam$. We prove that $h \circ g \colon \Gam \to \Gam$ is a PNP-free train track representative of some $\vphi \in \mA_r$ such that $\mG \cong \iwp$.

Notice that, since $G_1$ and $G_2$ have red edges indexed the same, the admissible decomposition of $g$ and $h$ (coming from their ideal decompositions) parse to give an admissible decomposition of $h \circ g$.

We show that $h \circ g$ is irreducible. Consider any pair of edges $E_i, E_j \in \mE(G)$. We show that $(h \circ g)^2(E_i)$ contains either $E_j$ or $\ol{E_j}$, noting that by Lemma \ref{L:LimitedWGs}B it suffices to show that $(h \circ g)(E_i)$ contains either $E_j$ or $\ol{E_j}$. We consider four (overlapping) cases separately. Suppose first that there exists an edge $E_i' \in \mE(\Gamma^1)$ indexed with $i$. If there exists an edge $E_j' \in \mE(\Gamma^1)$ indexed by $j$, then either $E_j'$ or $\ol{E_j}'$ is contained in $g(E_i)$ so, by Lemma \ref{L:LimitedWGs}B, either $E_j'$ or $\ol{E_j}'$ is in $(h \circ g)(E_i)$. Suppose there exists an edge $E_j' \in \mE(\Gamma^2)$ indexed by $j$. Since $g$ is strictly irreducible, for each $k \in \mI^+$, we have that $g(E_i)$ contains either $E_k$ or $\ol{E_k}$ and since $h$ is strictly irreducible, $h(E_k)$ contains either $E_j$ or $\ol{E_j}$. So $(h \circ g)(E_i)$ contains either $E_j$ or $\ol{E_j}$. Now suppose that there exists an edge $E_i' \in \mE(\Gamma^2)$ indexed with $i$. If there exists an edge $E_j' \in \mE(\Gamma^2)$ indexed by $j$, then $h(E_i')$ contains either $E_j'$ or $\ol{E_j}'$. But Lemma \ref{L:LimitedWGs}B implies $g(E_i)$ contains $E_i$, so that $(h \circ g)(E_i)$ contains either $E_j$ or $\ol{E_j}$. Suppose instead that there exists an edge $E_j' \in \mE(\Gamma^1)$ indexed by $j$. By Lemma \ref{L:LimitedWGs}B, we have that $g(E_i)$ contains $E_i$. Since $h$ is strictly irreducible, $h(E_i)$ contains either $E_k$ or $\ol{E_k}$ for any $E_k$ with $k \in \mI^+$, and since $g$ is strictly irreducible, $g(E_k)$ contains either $E_j$ or $\ol{E_j}$. Finally, Lemma \ref{L:LimitedWGs}B implies that $h(E_j)$ contains $E_j$. So $(h \circ g)^2(E_i)$ contains either $E_j$ or $\ol{E_j}$.

We now show that each edge of $\mG$ is labeled by a turn taken by some $(h \circ g)^p(E_i)$ where $p \geq 1$ and $E_i \in \mE(\Gamma)$. By construction, each edge $E$ of $\mG$ is either labeled according to a purple edge of $C(G_1)$ or is labeled according to a purple edge of $C(G_2)$. First suppose that $E$ is labeled according to a purple edge of $C(G_2)$ and corresponds to a turn $\{E_i, E_j\}$. Then $\{E_i, E_j\}$ is taken by some $h(E_k)$. By Lemma \ref{L:LimitedWGs}B, $g(E_k)$ contains $E_k$, so $\{E_i, E_j\}$ is taken by $(h \circ g)(E_k)$. Now suppose $E$ is labeled according to a purple edge of $C(G_1)$, so that $\{E_i, E_j\}$ is taken by some $g(E_k)$ and both $E_i$ and $E_j$ are fixed by $Dg$. Since $Dg$ and $Dh$ have the same nonfixed direction, $E_i$ and $E_j$ are both fixed by $Dh$. So, by Lemma \ref{L:LimitedWGs}, we have that $\{E_i, E_j\}$ is additionally taken by $(h \circ g)(E_k)$. Notice also that, since $Dg$ and $Dh$ each fix all periodic directions, and these are the same directions, that $G(h \circ g)$ has no purple edge not in $\mG$.

We are left to show that $h \circ g$ has no PNPs. Since $g$ is a legalizing Nielsen path prevention sequence when realized on $\Gamma^1$, Lemma \ref{l:LegalizingNPseq} implies that its extension to $\Gamma$ is still a Nielsen path prevention sequence. We then use the fact that $h \circ g \circ h$ is expanding and irreducible to deduce that $(h \circ g)^2$, and hence $h \circ g$, is PNP-free.
 \qedhere
\end{proof}

\section{Ideal Whitehead graphs with cut vertices}{\label{s:CutVertices}}

We restate the following theorem (\cite[Theorem A]{IWGIII}) because it can be used to build fully irreducibles in higher ranks with certain ideal Whitehead graphs. 

\begin{thm}{\label{T:iwgIIIMainTheorem1}}
Exactly eighteen of the twenty-one connected, simplicial five-vertex graphs are the ideal Whitehead graph $\mathcal{IW}(\vphi)$ for a fully irreducible outer automorphism $\vphi \in Out(F_3)$.

Those that are not the ideal Whitehead graph for any fully irreducible $\vphi\in Out(F_3)$ are precisely:
%~\\
%\vspace{-9mm}
\begin{figure}[H]
\centering
\includegraphics[width=2.3in]{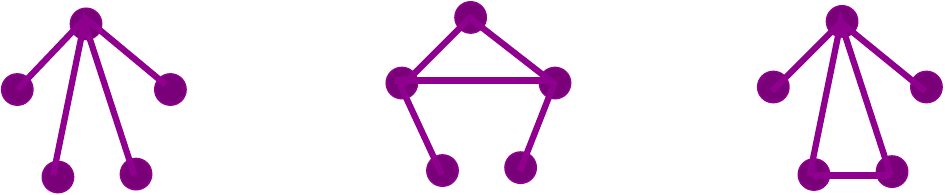}
\label{fig:UnachievableGraphs} %\\[-3mm]
\end{figure}
\end{thm}

The $\mathcal{ID}$ diagrams for the graphs of \cite[Theorem A]{IWGIII} are of particular use. Some of these can be found in \cite{IWGIII}. Instructions on how to construct them can be found in \cite{p12a}. We include here an $\mathcal{ID}$ diagram that we will specifically use in the proof of Theorem \ref{main1}:

\begin{figure}[H]
 \centering
 \noindent \includegraphics[width=5.5in]{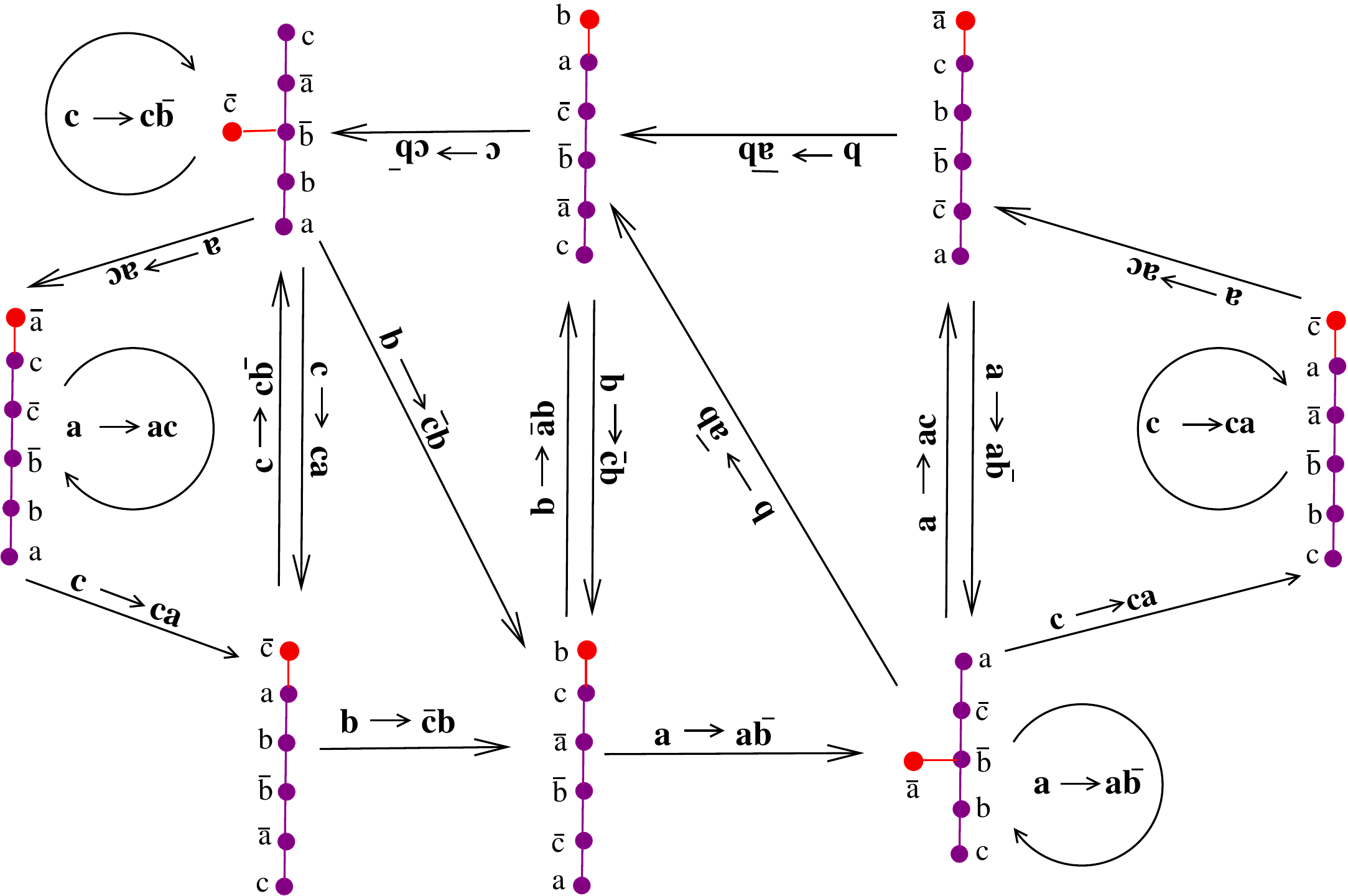} %\\[-4mm]
 \caption{$\mathcal{ID}(\mG)$ component where $\mG$ is the line}
 \label{fig:LineIDDiagram}
 \end{figure}
 
 \begin{mainthmA}\label{main1}
For each $r \geq 3$, there exists an ageometric fully irreducible $\vphi \in Out(F_r)$ such that the ideal Whitehead graph $\mathcal{IW}(\vphi)$ is a connected $(2r-1)$-vertex graph with at least one cut vertex. In particular, there exists an ageometric fully irreducible $\vphi \in Out(F_r)$ with the single-entry index list $\{\frac{3}{2}-r\}$ and such that the axis bundle of $\vphi$ has more than one axis.
\end{mainthmA}

\begin{proof}

Let $\mG$ be the graph such that Figure \ref{fig:LineIDDiagram} is a component of $\mathcal{ID}(\mG)$. The proof of Theorem 5.1 of \cite{IWGIII} gives at least one loop $L$ in the $\mathcal{ID}$ diagram of Figure \ref{fig:LineIDDiagram} that yields an ageometric fully irreducible $\vphi \in Out(F_r)$ such that $\mathcal{IW}(\vphi) \cong \mG$. Hence, such a loop exists for each node in the diagram, including the left-most node, whose ltt structure we call $G_1$, and the right-most node of the upper 3, which we call $G_2$. In fact, since the loop $L$ contains $G_1$ and $G_2$, we can use the same loop (starting at different nodes) to obtain $G_1$ and $G_2$. We call these new loops, respectively, $L_1$ and $L_2$. After taking adequately high powers of these loops, any cyclically admissible sequence containing the sequence from either of these (possibly necessarily iterated) loops will be PNP-free, as Lemma \ref{l:SpecificLegalizingNPseq} indicates that $L^2$ gives a legalizing NP prevention sequence and changing the marking does not change whether a train track map has Nielsen paths (see, for example, the proof of \cite[Lemma 3.1]{IWGI}). Applying Theorem B once to $G_1$ and $G_2$ gives:

\begin{figure}[H]
 \centering
 \noindent \includegraphics[width=1in]{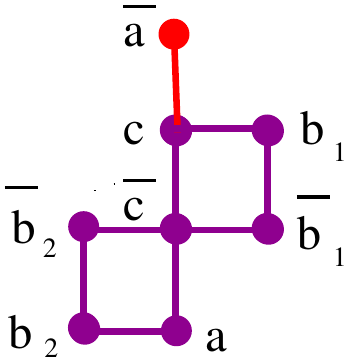} %\\[-4mm]
 \label{fig:IWG4_1}
 \end{figure}
 
 And iterated application gives
 
 \begin{figure}[H]
 \centering
 \noindent \includegraphics[width=1.85in]{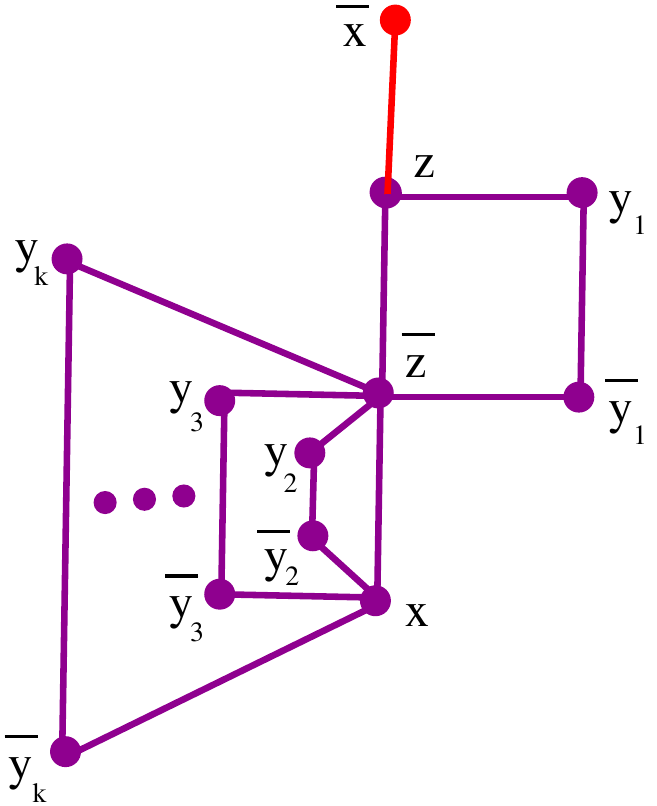} %\\[-4mm]
 \label{fig:IWG4_2}
 \end{figure}
\end{proof}

The cut vertex is at $\bar{z}$.

%\newpage

\bibliographystyle{amsalpha}
\bibliography{PaperReferences}
%\bibliography{PaperReferencesIWG}

\end{document}